\documentclass[12pt,a4paper,reqno]{amsart}
\usepackage[mathcal]{eucal}
\usepackage{amssymb}
\usepackage[nice]{nicefrac}
\usepackage{hyperref}
\usepackage[normalem]{ulem}
\usepackage{xcolor}

\usepackage{amsthm}
\usepackage{mathtools}
\usepackage{amsfonts}
\usepackage{comment}
\usepackage{faktor}

\theoremstyle{plain}
\newtheorem{theorem}{Theorem}[section]
\newtheorem{lemma}[theorem]{Lemma}
\newtheorem{proposition}[theorem]{Proposition}
\newtheorem{corollary}[theorem]{Corollary}

\theoremstyle{definition}
\newtheorem{definition}[theorem]{Definition}

\numberwithin{equation}{section}

\makeatletter
\newcommand{\AlignFootnote}[1]{%
    \ifmeasuring@
    \else
        \footnote{#1}%
    \fi
}
\makeatother

\title{Finite groups in which every commutator has prime power order}

\author[M. Figueiredo]{Mateus Figueiredo} 
\address{Mateus Figueiredo: Department of Mathematics, University of Brasilia, Brasilia DF, Brazil}
\email{mt-figueiredo@hotmail.com}

\author[P. Shumyatsky]{Pavel Shumyatsky} 
\address{Pavel Shumyatsky: Department of Mathematics, University of Brasilia, Brasilia DF, Brazil}
\email{pavel@unb.br}

\thanks{This work was supported by CNPq.}


\begin{document}

\maketitle

\begin{abstract}
Finite groups in which every element has prime power order (EPPO-groups) are nowadays fairly well understood. For instance, if $G$ is a soluble EPPO-group, then the Fitting height of $G$ is at most 3 and $|\pi(G)|\leqslant2$ (Higman, 1957). Moreover, Suzuki showed that if $G$ is insoluble, then the soluble radical of $G$ is a 2-group and there are exactly eight nonabelian simple EPPO-groups. 

\noindent In the present work we concentrate on finite groups in which every commutator has prime power order (CPPO-groups). Roughly, we show that if $G$ is a CPPO-group, then the structure of $G'$ is similar to that of an EPPO-group. In particular, we show that the Fitting height of a soluble CPPO-group is at most 3 and $|\pi(G')|\leqslant3$. Moreover, if $G$ is insoluble, then $R(G')$ is a 2-group and $G'/R(G')$ is isomorphic to a simple EPPO-group. 
\end{abstract}

\section{Introduction}

Finite groups in which every element has prime power order (EPPO-groups for short) were first studied by Higman in \cite{higman}. In the literature these groups are sometimes  called CP-groups. Higman showed that a soluble EPPO-group has Fitting height at most 3. Moreover, the order of a soluble EPPO-group is divisible by at most two primes. Suzuki classified simple EPPO-groups in his celebrated work \cite{Suzuki}, finding that only eight simple EPPO-groups exist. These are the groups $PSL(2,q)$ $(q = 4,7,8, 9, 17)$, $PSL(3, 4)$, $Sz(8)$, $Sz(32)$.  Moreover Suzuki showed in \cite{suzuki0} that if $G$ is an insoluble EPPO-group, then the soluble radical $R(G)$ is a 2-group. Further clarifications of the structure of EPPO-groups were obtained in \cite{brandl} and \cite{tiedt}.

More recently, also infinite groups in which every element has prime power order have attracted attention. In particular, the reader can check the papers \cite{shu} and \cite{delgado} for the study of profinite and locally finite groups with that property.

In this paper we focus on finite groups in which the commutators have prime power orders (CPPO-groups). By a commutator we mean any element $a$ of a group $G$ for which there are $x,y\in G$ such that $a=[x,y]=x^{-1}y^{-1}xy$. As usual, $G'$ stands for the commutator subgroup of $G$. It is well-known that elements of $G'$ need not be commutators. On the other hand, by the celebrated verification of the Ore conjecture \cite{Oreconject}, every element of a nonabelian finite simple group is a commutator. Therefore any simple CPPO-group is an EPPO-group (from the above list determined by Suzuki). We do not know if the commutator subgroup of any CPPO-group is necessarily EPPO. This seems unlikely. Our main results can be summarized as follows.

\begin{theorem}\label{main1}
	Let $G$ be a soluble CPPO-group. Then
	\begin{enumerate}
		\item[(a)] The Fitting height of $G$ is at most $3$;
		\item[(b)] The order of $G'$ is divisible by at most $3$ primes.
	\end{enumerate}
\end{theorem}
As usual, we say that a group $K$ is perfect if $K=K'$. We write $R(K)$ to denote the maximal soluble normal subgroup of $K$.
\begin{theorem}\label{main2}
	Let $G$ be an insoluble CPPO-group. Then
	\begin{enumerate}
		\item[(a)] $G'$ is perfect;
		\item[(b)] $R(G')=[G',R(G)]$ is a $2$-group; and 
\item[(c)]		$G'/R(G')$ is a simple EPPO-group.	\end{enumerate}
\end{theorem}

Thus, our results show that the structure of the commutator subgroup of a CPPO-group is roughly similar to that of an EPPO-group.

%In view of this the following problem seems rather intriguing. 
%\medskip

%{\it Does there exist a perfect CPPO-group, which is not an EPPO-group?}
%\medskip

%So far we could neither prove nor disprove the existence of such a group.}

\section{Preliminaries}

If $A$ is a group of automorphisms of a group $G$, the subgroup generated by all elements of the form $g^{-1}g^\alpha$ with $g\in G$ and $\alpha\in A$ is denoted by $[G,A]$. It is well known that the subgroup $[G,A]$ is an $A$-invariant normal subgroup in $G$. We write $C_G(A)$ for the centralizer  of $A$ in $G$. If $G$ and $A$ are finite and $(|G|,|A|)=1$, we say that $A$ is a group of coprime automorphisms of $G$. Throughout, $\pi(G)$ denotes the set of prime divisors of the order of $G$.

We start with a lemma which lists some well known properties of coprime actions (see for example \cite[Ch.~5 and 6]{gorenstein2007finite}).  In the sequel the lemma will often be used without explicit references.

\begin{lemma}\label{cc}
Let  $A$ be a group of coprime automorphisms of a finite group $G$. Then
\begin{enumerate}
\item[(i)] $G=[G,A]C_{G}(A)$. If $G$ is abelian, then $G=[G,A]\oplus C_{G}(A)$.
\item[(ii)] $[G,A,A]=[G,A]$. 
\item[(iii)] $C_{G/N}(A)=NC_G(A)/N$ for any $A$-invariant normal subgroup $N$ of $G$.
\item[(iv)] If $[G/\Phi(G),A]=1$, then $[G,A]=1$.
\item[(v)] If $G$ is nilpotent and $A$ is a noncyclic abelian group, then $G=\prod_{1\neq a\in A}C_{G}(a)$.
\item[(vi)] $G$ contains an $A$-invariant Sylow $p$-subgroup for each prime $p\in\pi(G)$.
\end{enumerate}
\end{lemma}

The next lemma is immediate from the previous one.
\begin{lemma}\label{fundresultkurzweil}
	Let $A$ be a group acting coprimely on a finite abelian group $V\neq1$. Suppose that  
	$C_V(a)=1,~\text{for every}~1\neq a\in A$.
	Then $A$ is cyclic if at least one of the following conditions holds:
	\begin{enumerate}
		\item $A$ is abelian;
		\item $A$ is a $p$-group for an odd prime $p$;
		\item $A$ is a $2$-group but not a quaternion group. 
	\end{enumerate}
\end{lemma}

\begin{lemma}\label{acnoncop}
	Let $A$ be a noncyclic abelian group acting coprimely on an abelian group $V$. Then 
	$$\bigcap_{1\neq a\in A}[V,a]=1.$$
\end{lemma}
\begin{proof}
	We argue by induction on $|V|$. The result is clear if $|V|=1$. Assume $|V|>1$. As $A$ is a noncyclic group, by Lemma \ref{fundresultkurzweil}  there exists an element $1\neq a\in A$ for which $C_V(a)\neq1$. Since $A$ is abelian, $C_V(a)$ is $A$-invariant. Applying the induction hypothesis to the quotient group $V/C_V(a)$ we obtain
	$$\bigcap_{1\neq b\in A}\big[V/C_V(a),b\big]=1.$$
	This means that $C_V(a)$ contains the subgroup 
	$$\bigcap_{1\neq b\in A}[V,b],$$
	which clearly is also contained in $[V,a]$. Since the action is coprime,
	$$[V,a]\cap C_V(a)=1,$$ 
	whence 
	$$\bigcap_{1\neq b\in A}[V,b]=1,$$
	as claimed. 
\end{proof}

Throughout this article we write $G=\langle X\rangle$ to mean that the group $G$ is generated by a set $X$.

\begin{lemma}\label{orderofav}
	Let $G=V\langle a\rangle$ be a finite group, which is a product of a normal subgroup $V$ and a cyclic subgroup $\langle a\rangle$ such that $(|V|,|a|)=1$. Suppose that $v\in V$ has the property that $(|V|,|av|)=1$. Then $v\in [V,a]$. 
\end{lemma} 
\begin{proof} We can pass to the quotient group $G/[V,a]$ and without loss of generality assume that $[V,a]=1$. But then it is clear that $(|V|,|av|)=1$ if and only if $v=1$. Hence the result.
\end{proof}

We remind the reader that a  $p$-group $P$ is called extraspecial if the centre $Z(P)$ has order $p$, and the quotient $P/Z(P)$ is a nontrivial elementary abelian $p$-group. 

\begin{lemma}\label{autoofextra}
	Let $\varphi$ be a coprime automorphism of a finite extraspecial $p$-group $P$ such that $C_P(\varphi)=\Phi(P)$. Then every element of $P\setminus \Phi(P)$ is conjugate to an element of the form $[x,\varphi]$ for some $x\in P$. 
\end{lemma}
\begin{proof} Since $C_P(\varphi)=\Phi(P)$, it follows that the map $\overline{g}\mapsto[\overline{g},\varphi]$ is surjective on the group $\overline{P}=P/\Phi(P)$. Thus, for an element $g\in P\setminus\Phi(P)$ we can find elements $x\in P\setminus\Phi(P)$ and $z\in\Phi(P)$ for which $g=[x,\varphi]z$. Let $y\in P$ be an element such that $[x,\varphi,y]\neq1$. Since $\Phi(P)=Z(P)$ has prime order, it is generated by $[x,\varphi,y]$. Therefore there exists an integer $r$ satisfying $[x,\varphi,y]^r=z^{-1}$. It follows that 
	$$g^{y^r}=([x,\varphi]z)^{y^r}=[x,\varphi]^{y^r}z=[x,\varphi][x,\varphi,y^r]z=[x,\varphi][x,\varphi,y]^rz=[x,\varphi].$$ This shows that $g$ is conjugate to $[x,\varphi]$, as required.
\end{proof}

 An important role in this paper is played by the concept of towers as introduced by Turull (see \cite{Turull}).

\begin{definition}\label{DefTurulltower}
	Let $G$ be a group. A sequence $(P_{i})_{i=1,\ldots,h}$ of subgroups of $G$ is said to be a tower of height $h$ if the following are satisfied:
	\begin{enumerate}
		\item $P_{i}$ is a $p_i$-group for all $i=1,\ldots,h$.
		\item $P_{i}$ normalizes $P_j$ for all $i<j$.
		\item Put $\overline{P_h}=P_{h}$ and $\overline{P_{i}}={P_{i}}/{C_{P_{i}}(\overline{P_{i+1}})}, \  i=1,\ldots,h-1.$ 
		Then $\overline{P_i}$ is nontrivial for all $i$.
		\item $p_i\neq p_{i+1}$ for all $i=1,\ldots,h-1$.
	\end{enumerate}
\end{definition}

In what follows, whenever we have a tower $(P_{i})_{i=1,\ldots,h}$, we write $p_i$ to denote the prime divisor of the order of $P_i$.

Recall that the Fitting height of a finite soluble group is the minimal number $h=h(G)$ such that $G$ possesses a normal series of length $h$ all of whose factors are nilpotent. The next lemma lemma is taken from \cite{Turull}.
\begin{lemma}\label{heightequaltower} If $G$ is a finite soluble group, then $h(G)$ is the maximum of heights of towers of $G$.
\end{lemma}
The following lemma will be useful. It is taken from \cite[Lema 3.7]{Casolo}.
\begin{lemma}\label{Casoloquotienttower}
	Let $(P_i)_{i=1,\ldots,h}$ be a tower of a group $G$ with $h\geqslant2$. Let $N$ be anormal subgroup of $G$ such that
	$$N\cap P_i\leqslant C_{P_{i}}(P_h),~i=1,\ldots,h-1.$$ Then $(P_iN/N)_{i=1,\ldots,h-1}$ is a tower of $G/N$. 
\end{lemma}
The next lemma is given without a proof because it is pretty similar to \cite[Lema 3.6]{Casolo}.
\begin{lemma}\label{centralizadordeumtermo}
Let $(P_i)_{i=1,\ldots,h}$ be a tower of a group $G$. Then $C_{P_i}(P_j)\leq C_{P_i}(\overline{P_{i+1}})$ for each $1\leqslant i<j\leqslant h$.
\end{lemma}

We will now record the almost obvious but important fact that only the first two factors of a tower can be cyclic.
\begin{lemma}\label{P3 e nao ciclico}
	Let $(P_{i})$ be a tower of height $h\geqslant3$ in a group $G$. Then for any $i\geqslant3$ the subgroup $P_{i}$ is not cyclic. 
\end{lemma}
\begin{proof}
	Suppose that $P_i$ is a cyclic group for some $i\geqslant3$. Then any two automorphisms of $P_i$ commute. Consequently, $[P_{i-1}, P_{i-2}]\leq C_{P_{i-1}}(P_i)\leq C_{P_{i-1}}(\overline{P_i})$. It follows that $\overline{P_{i-2}}=1$, which is a contradiction. 
\end{proof}

We will require Turull's concept of an irreducible tower.
\begin{definition}\label{DefirredTurulltower}
	Let $G$ be a group and let $(P_i)$ be a tower of height $h$ in $G$. The tower $(P_i)$ is said to be irreducible if the following are satisfied:
	\begin{enumerate}
		\item $\Phi(\Phi(\overline{P_i}))=1$, $\Phi(\overline{P_i})\leq Z(\overline{P_i})$ and if $p_i\neq 2$ then $exp(\overline{P_i})=p_i$ for $i=1,\ldots,h$ and $P_{i-1}$ centralizes $\Phi(\overline{P_i})$ for all $i=2,\ldots,h$. 
		\item $P_1$ is cyclic and $\overline{P_1}$ has prime order.
		\item There exists an elementary abelian subgroup $H_{i}$ of $\overline{P_{i-1}}$ such that $[H_i,\overline{P_i}]=\overline{P_i}$ for $i=2,\ldots,h$.
		\item If $H$ is a $P_1\cdots P_{i-1}$-invariant subgroup of $P_i$ whose image on $\overline{P_i}$ is not contained in $\Phi(\overline{P_i})$, then $H=P_i$. 
	\end{enumerate}
\end{definition}

	Let $(P_{i}^{(1)})$ and $(P_{j}^{(2)})$ be towers of a group $G$ of heights $h_1$ and $h_2$, respectively. We will say that $(P_{i}^{(1)})$ is contained in $(P_{j}^{(2)})$ if there exists an increasing map $f:\{1,\ldots,h_1\}\longrightarrow\{1,\ldots,h_2\}$ such that $P_{i}^{(1)}\subseteq P_{f(i)}^{(2)}$ for every $i=1,\ldots,h_1$. 
	
Note that our definition of irreducible towers differs from that given by Turull in \cite{Turull} only by the item (ii). However, the difference is inessential. Indeed, if $(P_1,\ldots, P_h)$ is a tower of a group $G$, we may take an element $a\in P_1\setminus C_{P_1}(\overline{P_2})$ so that $(\langle a\rangle, P_2,\ldots,P_h)$ is also a tower of $G$. By \cite[Lemma 1.4]{Turull},  this last tower contains an irreducible one which now satisfies Definition \ref{DefirredTurulltower}.  Therefore we have

\begin{lemma}\label{containirredtower}
Let $G$ be a group and let $(P_i)_{i=1,\ldots,h}$ be a tower of $G$. Then $(P_i)$ contains an irreducible tower of same height. 
\end{lemma}

\section{Soluble CPPO-groups}

Observe that subgroups and quotient groups of a CPPO-group are again CPPO-groups. In the sequel, this will be used throughout the paper without being mentioned explicitly.

\begin{lemma}\label{aaa}
	Let $G$ be a finite group containing a tower $(P_1,P_2,P_3)$ of abelian subgroups such that $P_1$ is cyclic, $P_2$ is noncyclic and $P_2=[P_2,P_1]$. Then $G$ has a commutator whose order is not a prime power. 
\end{lemma}
\begin{proof} Suppose the lemma is false and $G$ is a CPPO-group.

	Write $P_1=\langle a\rangle$. For elements $1\neq b\in P_2$ and $c\in P_3$, the equality
	$$[c,a][a,b]=[cb^{-1},a]^{b}$$
implies that $[c,a][a,b]$ has prime power order. Therefore, Lemma \ref{orderofav} guarantees that 
	$[c,a]\in [P_3,[a,b]]$, which is equivalent to $[c,a]\in [P_3,[b,a]]$. Since $c$ can be chosen in $P_3$ arbitrarily, we get
	\begin{equation}\label{P3P1}
		[P_3,P_1]=[P_3,a]\leq [P_3,[b,a]],~\text{for any}~1\neq b\in P_2.
	\end{equation}
	On the other hand, $P_2$ is an abelian group and $C_{P_2}(a)=1$. It follows that the map $b\mapsto[b,a]$ is surjective on $P_2$. Therefore the containment \ref{P3P1} implies that
	\begin{equation}\label{P3P11}
		[P_3,P_1]\leq\bigcap_{1\neq b\in P_2}[P_3,b].
	\end{equation}
	We conclude from (\ref{P3P11}) and from Lemma \ref{acnoncop} that $[P_3,P_1]=1$, that is, $P_1=C_{P_1}(P_3)$. In view of Lemma \ref{centralizadordeumtermo} we conclude that $\overline{P_1}=1$, which contradicts the definition of tower. 
\end{proof}

\begin{lemma}\label{torrecasoextra}
	Let $G$ be a CPPO-group containing a tower $(P_1,P_2,P_3)$ with the following properties:
	\begin{enumerate}
		\item $P_1$ is cyclic;
		\item $P_2$ is extraspecial and $C_{P_2}(P_1)=\Phi(P_2)$;
		\item $P_3$ is abelian and $P_3=[P_3,\Phi(P_2)]$.
	\end{enumerate}
	Then $p_2=2$ and $P_2$ is isomorphic to the quaternion group $Q_8$. 
\end{lemma}
\begin{proof} 
Write $P_1=\langle a\rangle$. For any elements $b\in P_2\setminus\Phi(P_2)$ and $c\in P_3$, the element $[c,a][a,b]=[cb^{-1},a]^b$ has prime power order, so by Lemma \ref{orderofav} we have \begin{equation}\label{af1}
		[c,a]\in [P_3,[b,a]].\end{equation} 
Moreover, for a nontrivial element $z\in \Phi(P_2)$, the element $[c,z][c,a]^z[a,b]=[cb^{-1},az]^b$ has prime power order and using again Lemma \ref{orderofav} we see that 
	\begin{equation}\label{af2}
		[c,z][c,a]^z\in [P_3,[b,a]].\end{equation} 
	Since $P_2$ is extraspecial, $z$ commutes with $[b,a]$ and it follows from (\ref{af1}) and (\ref{af2}) that
	$$[c,z]\in [P_3,[b,a]].$$
	By hypothesis we have $P_3=[P_3,\Phi(P_2)]$. Since the element $c$ can be chosen in $P_3$ arbitrarily, we conclude that
	\begin{equation}\label{af3}
		P_3=[P_3,[b,a]],~\text{for any}~b\in P_2\setminus\Phi(P_2).\end{equation}
Lemma \ref{autoofextra} guarantees that any element of $P_2\setminus\Phi(P_2)$ may be writen in the form $[b,a]^x$ for some $b\in P_2\setminus\Phi(P_2)$ and $x\in P_2$. We deduce from (\ref{af3}) and the equality $P_3=[P_3,\Phi(P_2)]$ that 
	$$C_{P_3}(b)=1,~\text{for any}~1\neq b\in P_2.$$
By virtue of Lemma \ref{fundresultkurzweil} we deduce that $p_2=2$ and $P_2$ is a quaternion group. Taking into account that $Q_8$ is the only extraspecial quaternion group, $P_2$ is isomorphic to $Q_8$. 
\end{proof}

\begin{lemma}\label{abelianorextra}
	Let $G$ be a CPPO-group and suppose that $G$ contains an irreducible tower $(P_1,P_2,P_3,P_4)$ with the following properties:
	\begin{enumerate}
		\item $C_{P_3}(P_4)=1$;
		\item $P_4$ is elementary abelian.
	\end{enumerate}
Let $i\in\{2,3\}$.	Then $\overline{P_i}$ is either abelian or extraspecial. 
\end{lemma}
\begin{proof}
Note that, as the tower is irreducible, it is sufficient to check that $\Phi(\overline{P_i})$ is cyclic. In fact, the inverse image of $Z(\overline{P_i})$ is a $P_1\cdots P_{i-1}$-invariant subgroup of $P_i$. Thus if $Z(\overline{P_i})\nleq \Phi(\overline{P_i})$, Definition \ref{DefirredTurulltower} ensures that $\overline{P_i}=Z(\overline{P_i})$ is an abelian group. On the other hand, $\Phi(\Phi(\overline{P_i}))=1$ and $\Phi(\overline{P_i})\leq Z(\overline{P_i})$. In short, $\Phi(\overline{P_i})$ is elementary abelian and either $\Phi(\overline{P_i})=Z(\overline{P_i})$ or $\overline{P_i}$ is abelian. 
	
First, we prove that $\Phi(\overline{P_2})$ is cyclic. Let $K$ be the kernel of the action of $P_1P_2$ on $P_3/\Phi(P_3)$. Observe that $P_1P_2/K$ acts faithfully and irreducibly on $P_3/\Phi(P_3)$, so that $P_1P_2/K$ has cyclic centre. We will show that $\Phi(\overline{P_2})$ is a central subgroup of $P_1P_2/K$. This would guarantee that $\Phi(\overline{P_2})$ is cyclic. Let $Q$ be the inverse image of $\Phi(\overline{P_2})$ in $P_2$. From Definition \ref{DefirredTurulltower} we know that $$[Q,P_1P_2]\leq C_{P_2}(P_3)\leq K,$$ and consequently, 
	$$[P_1P_2/K,\Phi(P_2K/K)]=1.$$
	This means that $\Phi(\overline{P_2})$ is a central subgroup of $P_1P_2/K$, as claimed. 
	
Now we will prove that $\Phi(P_3)$ is cyclic. Let $L$ be the kernel of the action of $P_1P_2P_3$ over $P_4$. Observe that $P_1P_2P_3/L$ acts faithfully and irreducibly on $P_4$ and consequently has cyclic centre. Therefore, as $P_3\cap  L=1$, it is sufficient to check that $\Phi(P_3)\leq Z(P_1P_2P_3)$. But if this is not the case, since $\Phi(P_3)$ is centralized by both $P_2$ and $P_3$, we can find elements $a\in P_1$ and $c\in \Phi(P_3)$ for which $[c,a]\neq1$. Clearly we may assume that $P_1=\langle a\rangle$ and find an element $b\in P_2$ such that $[b,a]\neq1$. It follows that $[cb,a]=[c,a][b,a]$ has order divisible by $p_2$ and $p_3$, which is impossible because $G$ is a CPPO-group. The proof is now complete.
\end{proof} 

\begin{lemma}\label{abelianoextra2}
	Let $G$ be a CPPO-group having an irreducible tower $(P_1,P_2,P_3,P_4)$ such that
	\begin{enumerate}
		\item $\overline{P_2}$ is extraspecial;
		\item $C_{P_3}(P_4)=1$;
		\item $P_4$ is elementary abelian.
	\end{enumerate}
	Then $p_2=2$ and $\overline{P_2}$ is isomorphic to $Q_8$. 
\end{lemma}
\begin{proof}
	Set $N=C_{P_2}(P_3)\Phi(P_3)$ and observe that
	$$C_{P_2N/N}(P_3N/N)=1.$$
	It follows from Lemma \ref{Casoloquotienttower} that $(P_iN/N)_{i=1,2,3}$ is a tower of $P_1P_2P_3/N$. Since
	$$[P_2N/N,P_1N/N]=P_2N/N,$$
	it follows that 
	\begin{equation}\label{bf1}
		C_{P_2N/N}(P_1N/N)\leq\Phi(P_2N/N).\end{equation}
	 Let $Q$ be the inverse image of $\Phi(\overline{P_2})$. By the irreducibility of the tower $(P_1,P_2,P_3,P_4)$ we have 
	$$[Q,P_1]\leq C_{P_2}(P_3)\leq N,$$ 
	and so
	\begin{equation}\label{bf2}
		[\Phi(P_2N/N),P_1N/N]=1.
	\end{equation}
	It follows from (\ref{bf1}) and (\ref{bf2}) that 
	\begin{equation}\label{bf3}
		C_{P_2N/N}(P_1N/N)=\Phi(P_2N/N). 
	\end{equation}
Observe that $Q\unlhd P_1P_2$ and hence $[P_3,Q]$ is a $P_1P_2$-invariant subgroup of $P_3$ which is not contained in $\Phi(P_3)$ since $\overline{P_2}$ is extraspecial. The irreducibility of the tower $(P_1,P_2,P_3,P_4)$ now shows that $P_3=[P_3,Q]$. Therefore we have
	$$P_3N/N=[P_3N/N,\Phi(P_2N/N)].$$
	
Note that $P_2N/N\cong\overline{P_2}$ and so the tower $(P_iN/N)_{i=1,2,3}$ satisfies the hypotheses of Lemma \ref{torrecasoextra}. Hence $p_2=2$ and $P_2N/N$ is isomorphic to $Q_8$. The proof is complete. 
\end{proof}

The following elementary observation will be helpful later on.

\begin{lemma}\label{autodoquaternion}
Let $\varphi$ be an involutory automorphism of a group $G$ isomorphic to $Q_8$. There exists an element $u\in G$ such that $[u,\varphi]$ is the involution of $G$. 
\end{lemma}
\begin{proof}
Suppose that this is false. Let $u\in G\setminus\Phi(G)$ be such that $u^\varphi\Phi(G)=u\Phi(G)$. We have $[u,\varphi]\in \Phi(G)$ and so $[u,\varphi]=1$, that is, $\langle u\rangle\leq C_G(\varphi)$. Since $\varphi$ is nontrivial, we get $\langle u\rangle=C_G(\varphi)$.
	Choose an element $x\in G\setminus C_G(\varphi)$. We certainly have $x^\varphi\neq x$. The assumption that the lemma is false implies that $x^\varphi\neq x^{-1}$. Let $y=x^\varphi$. As $\varphi$ is of order two, we have that $y^\varphi=x$. On the one hand, $(xy)^\varphi=x^\varphi y^\varphi=yx$. On the other hand, as $G=C_G(\varphi)\cup\{x,y,x^{-1},y^{-1}\}$, we have $xy\in C_G(\varphi)$ and so $xy=yx$. It follows that $x\in Z(G)$. This is a contradiction since $|Z(G)|=2$. 
\end{proof}

We are now ready to prove that the Fitting height of a soluble CPPO-group is at most three.
\begin{proposition}\label{a1}
	Let $G$ be a soluble CPPO-group. Then $h(G)\leqslant3$. 
\end{proposition}
\begin{proof}
Assume that the result is false and let $G$ be a counterexample of minimal possible order. By minimality, $h(G)=4$. Moreover, $\Phi(G)=1$ since $h(G/\Phi(G))=h(G)$. 
	
	Using Lemmas \ref{heightequaltower} and \ref{containirredtower} we may write $G=P_1P_2P_3P_4$ where $(P_1,P_2,P_3,P_4)$ is an irreducible tower. Here $P_4$ is normal in $G$ and so $\Phi(P_4)\leq\Phi(G)=1$, that is, $P_4$ is an elementary abelian $p_4$-group. Putting $N=C_{P_3}(P_4)$, observe that $N$ is normal and 
	$$C_{P_3/N}(P_4N/N)=1.$$
It follows from Lemma \ref{Casoloquotienttower} that $(P_iN/N)_{i=1,2,3,4}$ is a tower of $G/N$. Because of Lemma \ref{heightequaltower} and the minimality of $G$ we conclude that $N=1$. Hence, $P_3=\overline{P_3}$. Moreover, Lemma \ref{abelianorextra} shows that $\overline{P_i}$ is either abelian or extraspecial for $i\in\{2,3\}$ and $\Phi(P_3)\leq Z(P_1P_2P_3)$.
	
The remaining part of the proof consists in analysis of the following three possibilities for the group $\overline{P_2}$: either $\overline{P_2}$ is cyclic, or abelian noncyclic, or extraspecial. The proof will be complete once we show that in all these cases $G$ contains a commutator whose order is not a prime power.
	
	\textbf{Case 1.} Assume that $\overline{P_2}$ is a cyclic group and write $\overline{P_2}=\langle \overline{b}\rangle$. Observe that $P_3=[P_3,b]$ and $C_{P_3}(b)=\Phi(P_3)$. Thus if $P_3$ is an abelian group, Lemma \ref{aaa} shows that the subgroup $P_2P_3P_4$ contains a commutator whose order is not a prime power, a contradiction. Therefore $P_3$ is not abelian. Lemma \ref{abelianorextra} now shows that $P_3$ is an extraspecial $p_3$-group. Observe that $[P_4,\Phi(P_3)]$ is a normal subgroup contained in $P_4$. Moreover, $[P_4,\Phi(P_3)]\neq1$ since $C_{P_3}(P_4)=1$. It follows from the irreducibility of the tower that $P_4=[P_4,\Phi(P_3)]$. Now, applying Lemma \ref{torrecasoextra} to the tower $(P_2,P_3,P_4)$ we deduce that $p_3=2$ and $P_3\cong Q_8$. Consequently, we have $p_2=3$ and $p_1=2$. 
	
Set $K=C_{P_1P_2}(P_3)$. As neither $P_1$ nor $P_2$ acts trivially on $P_3$, we have that $P_1P_2/K$ is isomorphic to a subgroup of $S_4$ whose order is at least 6. Since the Sylow 2-subgroup of $P_1P_2/K$ is cyclic, it follows that $|P_1P_2/K|=6$, and in particular, $P_{1}^{2}\leq K$. Let $a\in P_1$ be a generator of $P_1$. Note that $a$ induces an involutory automorphism of $P_3$. Lemma \ref{autodoquaternion} now shows that there exists an element $u\in P_3$ for which $[u,a]$ is the involution of $P_3$. Since $\overline{P_2}=\langle\overline b\rangle$, we have $[b,a]\neq1$ and so $[ub,a]=[u,a][b,a]$ has order divisible by 6, a contradiction. 
	
	\textbf{Case 2.} Now we deal with the case where $\overline{P_2}$ is an abelian noncyclic group. Set $M=C_{P_2}(P_3)\Phi(P_3)$ and observe that
	$$
	C_{P_2M/M}(P_3M/M)=1.
	$$
In view of Lemma \ref{Casoloquotienttower} we deduce that $(P_iM/M)_{i=1,2,3}$ is a tower of $P_1P_2P_3/M$. 
	The irreducibility of the tower $(P_1,P_2,P_3,P_4)$ shows that 
	$$
	P_2M/M=[P_1M/M,P_2M/M].
	$$
	It follows from Lemma \ref{aaa} that $P_1P_2P_3/M$ has a commutator whose order is not a prime power, a contradiction.
	
	\textbf{Case 3.} It remains to handle the case where $\overline{P_2}$ is an extraspecial $p_2$-group. Lemma \ref{abelianoextra2} shows that $p_2=2$ and $\overline{P_2}$ is isomorphic to $Q_8$. 
	
Choose $b\in P_2$ such that $\Phi(\overline{P_2})=\langle\overline{b}\rangle$. We have $[P_3,b]=P_3$ and $C_{P_3}(b)=\Phi(P_3)$. Lemma \ref{aaa} shows that $P_3$ is not abelian and it follows from Lemma \ref{abelianorextra} that $P_3$ is an extraspecial $p_3$-group. Therefore $P_4=[P_4,\Phi(P_3)]$ and thus the tower $(\langle b\rangle,P_3,P_4)$ satisfies the hypotheses of Lemma \ref{torrecasoextra}. Since $p_3\neq2$, it follows that $\langle b\rangle P_3P_4$ contains a commutator whose order is not a prime power. This is the final contradiction.
\end{proof} 

As usual, if $\pi$ is a set of primes, we write $O_\pi(G)$ to denote the maximal normal $\pi$-subgroup of $G$.
\begin{lemma}\label{opelinha}
	Let $G$ be a CPPO-group and let $N$ be a nilpotent normal subgroup of $G$. Then there is a prime $p\in\pi(N)$ such that $O_{p'}(N)\leq Z(G)$. 
\end{lemma}
\begin{proof}
	If $N\leq Z(G)$ we have nothing to prove. Assume $N\nleq Z(G)$. There exists a prime number $p$ for which the Sylow $p$-subgroup $P$ of $N$ is not central in $G$. Then $G=C_{G}(P)\cup C_{G}(O_{p'}(N))$. Indeed, suppose $G\neq C_{G}(P)\cup C_{G}(O_{p'}(N))$ and choose $g\in G\setminus(C_{G}(P)\cup C_{G}(O_{p'}(N))) $. There are $a\in P$ and $b\in O_{p'}(N)$ such that $[a,g]\neq1$ and $[b,g]\neq1$. The equality $[ab,g]=[a,g][b,g]$ shows that $G$ contains a commutator whose order is not a prime power, a contradiction. Hence, $G=C_{G}(P)\cup C_{G}(O_{p'}(N))$. It is well-known that a nontrivial group cannot be a union of two proper subgroups. Since $P$ is not central, $G=C_G(O_{p'}(N))$, that is, $O_{p'}(N)\leq Z(G)$. 
\end{proof}

In what follows we write $\gamma_\infty(G)$ to denote the intersection of the lower central series of a group $G$.
\begin{proposition}\label{a2}
	Let $G$ be a soluble CPPO-group. Then $|\pi(G')|\leqslant3$. 
\end{proposition}
\begin{proof}
	If $G$ is a nilpotent group, Lemma \ref{opelinha} allows us to write $G=P\times O_{p'}(G)$ where $P$ is Sylow $p$-subgroup, for some prime $p$, and $O_{p'}(G)\leq Z(G)$. Then $G'\leq P$ is a $p$-group. 
	
	Assume that $h(G)=2$. In this case, $\gamma_\infty(G)$ is a nilpotent nontrivial normal subgroup of $G$. Using Lemma \ref{opelinha} we write $\gamma_\infty(G)=P\times O_{p'}(\gamma_\infty(G))$ where $P$ is a Sylow $p$-group of $\gamma_\infty(G)$ and $O_{p'}(\gamma_\infty(G))\leq Z(G)$. It follows that
	$$\gamma_\infty(G)=[\gamma_\infty(G),G]\leq P,$$
	that is, $\gamma_\infty(G)$ is a $p$-group. Moreover, $G/\gamma_\infty(G)$ is a nilpotent CPPO-group and applying the previous case we get that $G'/\gamma_\infty(G)$ is a $q$-group for some prime $q$. Thus $G'$ is a $\{p,q\}$-group. 
	
	The case $h(G)=3$ is obtained by applying twice the argument of the previous case. In fact, in this case we have $h(\gamma_\infty(G))=2$ and so $\gamma_\infty(\gamma_\infty(G))$ is a $p$-group for some prime $p$. But on the other hand we have $h(G/\gamma_\infty(\gamma_\infty(G)))=2$ and so $G'/\gamma_\infty(\gamma_\infty(G))$ has order divisible by at most two primes. Hence, the order of $G'$ is divisible by at most three primes.
\end{proof}

Note that the combination of Propositions \ref{a1} and \ref{a2} yields Theorem \ref{main1}.

\section{Insoluble CPPO-groups}

The goal of this section is to establish Theorem \ref{main2}. Our first lemma is almost obvious so we omit the proof.
\begin{lemma}\label{directproduct}
	Let $G=K\times L$ be a CPPO-group, where the subgroups $K$ and $L$ are both nonabelian. Then $G'$ is a $p$-group for some prime number $p$. 
\end{lemma}

Recall that the generalized Fitting subgroup $F^*(G)$ of a finite group $G$ is the product of the Fitting subgroup $F(G)$ and all subnormal quasisimple subgroups; here a group is quasisimple if it is perfect and its quotient by the centre is a nonabelian simple group. In any finite group $G$ we have $C_G(F^*(G))\leq F^*(G)$. Therefore the following lemma holds.

\begin{lemma}\label{kurosh}
Let $G$ be a nontrivial finite group with $R(G)=1$. Then $F^*(G)\neq 1$ and $C_G(F^*(G))=1$. 
\end{lemma}

Recall that a finite group is almost simple if it has a unique minimal normal subgroup (the socle) which is nonabelian simple. The  Mathieu group $M_{10}$ is an example of an almost simple EPPO-group, which is not simple.

\begin{lemma}\label{crofsemisimple}
Let $G$ be a nontrivial CPPO-group with $R(G)=1$. Then $G$ is almost simple.
\end{lemma}
\begin{proof}
Since $R(G)=1$, it follows that $F^*(G)$ is a direct product of nonabelian simple groups. Lemma \ref{directproduct} shows that $F^*(G)$ is simple. 
\end{proof}

The next theorem is the famous result, due to Liebeck et al, verifying the Ore conjecture.
\begin{proposition}\label{LiObShTi1}
Every element of a nonabelian simple group is a commutator. 
\end{proposition}
It follows that a simple group is a CPPO-group if and only if it is an EPPO-group. Combining this with Suzuki's classification of simple EPPO-groups we obtain.

\begin{proposition}\label{classsimpleCPPO}
	A nonabelian simple CPPO-group is isomorphic to one of the following groups: $PSL(2,q)$, with $q\in\{4,7,8,9,17\}$, $PSL(3,4)$, $Sz(8)$, $Sz(32)$. 
\end{proposition} 
As usual, $Out(G)$ denotes the outer automorphism group of $G$. The next lemma is now immediate (cf \cite{wilson}).
\begin{lemma}\label{outautogroup}
	Let $G$ be a nonabelian simple CPPO-group. Then one of the following conditions holds:
	\begin{enumerate}
		\item $Out(G)$ is cyclic;
		\item $G\cong PSL(2,9)\cong A_6$ and $Out(G)$ is the Klein four-group;
		\item $G\cong PSL(3,4)$ and $Out(G)$ is the dihedral group $D_{12}$, of order $12$.
	\end{enumerate}
\end{lemma}

 \begin{lemma}\label{psl34case}
Let $G$ be an almost simple group with the socle $H\cong PSL(3,4)$. Then $G$ is a CPPO-group if, and only if, $G/H$ is abelian.  
\end{lemma}
Before we embark on the proof, fix some notation. 

Let $F=\{0,1,a,a^2\}$ be the field with 4 elements and let $ L=SL(3,F)$. Denote by $\varphi$ the Frobenius automorphism of $L$ and note that $\varphi$ has order 2.

Let $\delta$ stand for the automorphism of $L$ induced by conjugation by the matrix
	\[
		\begin{bmatrix}
		a       & 0 & 0 \\
		0       & 1 & 0 \\
		0       & 0 & 1
	\end{bmatrix}\in GL(3,F).
	\]
Note that $\delta$ has order 3.

Denote by $\beta$ the diagonal automorphism of $L$, that is, the map taking a matrix $A\in L$ to the transpose of the inverse $A^{-T}$. Then $\beta$ has order 2. Observe that $\delta^\varphi=\delta^\beta=\delta^{-1}$. 

 Write  $\overline{\varphi}, \overline{\delta}, \overline{\beta}$ for the automorphisms of $H\cong PSL(3,4)$ induced by $\varphi,\delta$ and $\beta$, respectively.

 Set
	$G_1=H\langle\overline{\delta},\overline{\varphi}\rangle$
	and
$G_2=H\langle\overline{\delta},\overline{\beta}\rangle$.
	Thus, $G_1$ and $G_2$ are the two almost simple subgroups of $Aut\, H$  whose images in $Out\, H$ are nonabelian of order six.

We will now prove Lemma \ref{psl34case}. 

\begin{proof}[Proof of Lemma \ref{psl34case}]
Note that if $G/H$ is an abelian group, then $G'=H$. As $H$ is an EPPO-group, $G$ is a CPPO-group. 
	
Therefore it suffices to show that if $G/H$ is not abelian, then $G$ is not a CPPO-group. In this case, $G$ contains a subgroup isomorphic to either $G_1$ or $G_2$. Consequently, it suffices to show that $G_1$ and $G_2$ are not CPPO-groups. In the sequel we write $\overline X$ for the image of $X$ in $H$ whenever $X\subseteq L$.
	
We will show first that $G_1$ is not a CPPO-group. 
	
Let $A_1\in L$ be the matrix  
	$\begin{bmatrix}
		1&0&0\\
		0&1&a\\
		0&0&1
	\end{bmatrix}.
	$
Observe that $A_{1}^{2}=1$ and
	\[A_{1}^{\varphi}=\begin{bmatrix}
		1&0&0\\
		0&1&a^2\\
		0&0&1
	\end{bmatrix}.
	\]
We have
	\[[A_1,\varphi]=\begin{bmatrix}
		1&0&0\\
		0&1&1\\
		0&0&1
	\end{bmatrix}.
	\]
So $[A_1,\varphi]$ has order 2. Taking into account that $[A_1,\varphi]\in C_{ L}(\delta)$, remark that in $G_1$ the commutator
	$[\overline{A_1}\overline{\delta},\overline{\varphi}]=\overline{[A_1,\varphi]}\overline{\delta}$
	has order 6. This proves that $G_1$ is not a CPPO-group.
	
	Now we show that $G_2$ is not a CPPO-group.
	
	 Let $A_2\in L$ be the matrix  
	$\begin{bmatrix}
		1&0&0\\
		0&a&1\\
		0&0&a^2
	\end{bmatrix}.
	$
	Note that
	\[A_2^{\beta}=\begin{bmatrix}
		1&0&0\\
		0&a^2&0\\
		0&1&a
	\end{bmatrix}.
	\]
	Compute
	\[[A_2,\beta]=\begin{bmatrix}
		1&0&0\\
		0&a^2&1\\
		0&0&a
	\end{bmatrix}\begin{bmatrix}
		1&0&0\\
		0&a^2&0\\
		0&1&a
	\end{bmatrix}=\begin{bmatrix}
		1&0&0\\
		0&a^2&a\\
		0&a&a^2
	\end{bmatrix}.
	\]
Now it is easy to check that $[A_2,\beta]$ has order 2. Keeping in mind that $[A_2,\beta]\in C_Q(\delta)$ we conclude that in $G_2$ the commutator 
	$$[\overline{A_2}\overline{\delta},\overline{\beta}]
	=\overline{[A_2,\beta]}\overline{\delta}$$
	has order 6. Therefore $G_2$ is not a CPPO-group. This completes the proof.
\end{proof}

\begin{lemma}\label{crquotientofsemisimple}
Let $G$ be a nontrivial CPPO-group with $R(G)=1$. Then $F^*(G)=G'$ and $G/F^*(G)$ is either cyclic or the Klein four-group. 
\end{lemma}
\begin{proof}
By Lemma \ref{crofsemisimple} $F^*(G)$ is a nonabelian simple group. We have $C_G(F^*(G))=1$ and, identifying $F^*(G)$ with its group of inner automorphisms,  assume that $F^*(G)\leq G\leq Aut(F^*(G))$. If $F^*(G)\not\cong PSL(3,4)$, in view of Lemma \ref{outautogroup} we deduce that $G'=F^*(G)$ and $G/F^*(G)$ is either cyclic or  the Klein four-group. On the other hand, if $F^*(G)\cong PSL(3,4)$, it follows from Lemma \ref{psl34case} that $G'=F^*(G)$. Moreover, as $Out(PSL(3,4))\cong D_{12}$, we again deduce that $G/F^*(G)$ is either cyclic or the Klein four-group. 
\end{proof} 

The next observation is now straightforward.

\begin{corollary}
	\label{perfectplustrivialradical}
	Let $G$ be a nontrivial perfect CPPO-group with $R(G)=1$. Then $G$ is a simple EPPO-group. 
\end{corollary}

The following result is an immediate consequence of \cite[Theorem 1]{quasisimple}. 
\begin{lemma}\label{LiObShTi2}
	Let $G$ be a quasisimple group such that $G/Z(G)$ is a CPPO-group. One of the following statements holds.
	\begin{enumerate}
		\item Every element of $G$ is a commutator;
		\item $G/Z(G)\cong A_6$, $Z(G)$ is cyclic of order $3$ or $6$ and the noncentral elements of $G$ which are not commutators have orders in the set $\{12,15,24\}$;
		\item $G/Z(G)\cong PSL(3,4)$, $Z(G)\neq1$, $\pi(Z(G))\subseteq\{2,3\}$ and the noncentral elements of $G$ which are not commutators have orders divisible by $6$. 
	\end{enumerate}
\end{lemma}
Observe that the groups listed in Parts (ii) and (iii) of Lemma \ref{LiObShTi2} are not CPPO-groups. This is because in all cases there is an element $z\in Z(G)$ of order 2 or 3 and an element $g\in G$ of order 5 such that $zg$ is a commutator. Therefore we have the following result. 
\begin{lemma}\label{quasisimpleissimple}
If $G$ is a quasisimple CPPO-group, then $G$ is a simple EPPO-group. 
\end{lemma}
\begin{proof}  
Suppose $G$ is a quasisimple CPPO-group. In this case every element of $G$ is a commutator so that $G$ is actually an EPPO-group. Since in an EPPO-group every centralizer has a prime power order, the existence of a nontrivial centre implies that the group has prime power order. Thus, $Z(G)=1$ and so $G$ is simple.
\end{proof} 

\begin{lemma}\label{existelemabelqsub} Let $G$ be a finite group containing a normal subgroup $N$ such that $G/N$ is a nonabelian simple CPPO-group. Let $q\in\pi(G)\setminus\pi(N)$. There exists an elementary abelian $q$-subgroup $Q$ of $G$ and an element $a\in G$, whose order is a power of a prime different from $q$, such that $Q=[Q,a]$.
\end{lemma}
\begin{proof}
Remark that $G$ does not have a normal $q$-complement. Indeed, if $K$ is a normal $q$-complement and $G/K$ is a $q$-group, then $N\leq K$ and we get a contradiction since $G/N$ is simple. 
	
By the Frobenius theorem  \cite[Theorem 7.4.5]{gorenstein2007finite}, the group $G$ possesses a $q$-subgroup $H$ such that the group $N_G(H)/C_G(H)$ is not a $q$-group. Let $p\in\pi(N_G(H)/C_G(H))\setminus\{q\}$ and pick a $p$-element $a\in N_G(H)\setminus C_G(H)$. Thus $[H,a]\neq1$. Since $G/N$ is an EPPO-group, $a$ acts fixed-point-freely on $Z(H)$. Let $Q$ be the subgroup generated by elements of order $q$ of $Z(H)$. Note that $Q$ is elementary abelian and $Q=[Q,a]$. 
\end{proof}

Let $S_n$ and $A_n$ stand respectively for the symmetric and alternating groups on $n$ symbols. It is well-known that for $n\neq6$ every automorphism of $S_n$ is inner.
On the other hand, $S_6$ admits a nontrivial outer automorphism, often called exceptional. Note that the exceptional automorphism is not unique in the sense that if $\phi\in Aut\, S_6$ is exceptional and $y\in S_6$, then $\phi y\in Aut\, S_6\setminus Inn\, S_6$. Furthermore, $Aut\, S_6$ can be naturally identified with  $Aut\, A_6$ and $Out\, A_6$ is isomorphic to the Klein four-group. Slightly abusing terminology any automorphism of $S_6$, which is not inner, will be called exceptional.
\begin{lemma}\label{a6case} Let $\phi\in Aut\, S_6$ be an exceptional automorphism. There are elements $x\in S_6\setminus A_6$ and $y\in A_6$ such that $[x,\phi y]$ has odd order.
\end{lemma}
\begin{proof} Let $x=(123456)\in S_6\setminus A_6$. Up to a conjugation, we have $x^\phi=(142)(56)$. Let $y=(456)$. Observe that 
	$[x,\phi y]=x^{-1}y^{-1}x^\phi y=(143)(256),$
	and therefore $[x,\phi y]$ has order 3, as desired. 
\end{proof} 
\begin{proposition}\label{derivsubofnonsol}
The commutator subgroup of an insoluble CPPO-group is perfect. 
\end{proposition}
\begin{proof}
	Suppose that the result is false and let $G$ be a counterexample of minimal order. Clearly, $G$ is not perfect so the minimality of $|G|$ implies that $H=G''$ is a perfect group. Moreover, the minimality of $|G|$ also implies that $R(H)=1$. In short, $H$ is a nontrivial perfect CPPO-group with trivial soluble radical. Corollary \ref{perfectplustrivialradical} shows that $H$ is a nonabelian simple group. 
	
	Observe that $[R(G),H]\leq R(G)\cap H=1$, that is, $R(G)\leq C_G(H)$. We claim that $R(G)=Z(G)$. Suppose that this is not the case and take elements $a\in G$ and $b\in R(G)$ for which $[a,b]$ is a nontrivial $r$-element for some prime number $r$. 
	
Note that if $g\in H$ has the property that $ag$ normalizes an $r'$-subgroup $K$ of $H$, then $ag$ centralizes $K$. Indeed, for any $h\in K$ the equality
	\begin{equation}\label{cf1}
		[bh,ag]=[b,a][h,ag] 
	\end{equation}
	holds since $R(G)\leq C_G(H)$. It follows from (\ref{cf1}) that $[bh,ag]$ is an $r$-element. This implies that $[h,ag]$ is an $r$-element. On the other hand, $ag$ normalizes $K$ so that $[h,ag]\in K$ is an $r'$-element. It follows that $[h,ag]=1$ and so $[K,ag]=1$. 
	
We will now prove that $R(G)=Z(G)$. Let $q\in \pi(H)\setminus\{r\}$, and let $K$ be a nontrivial $q$-subgroup of $H$. Choose a Sylow $q$-subgroup $Q$ of $H$ containing $K$. By the Frattini argument there is an element $g\in H$ such that $ag\in N_G(Q)$. In view of the above we conclude that $ag\in C_G(Q)$ and so $ag\in C_G(K)$. For an  arbitrary element $x\in N_H(K)$ we have $agx\in N_G(K)$ and the above argument shows that $agx\in C_G(K)$. It follows that $x\in C_H(K)$. Thus, we proved that $N_H(K)=C_H(K)$ for any $q$-subgroup $K$ of $H$. The Frobenius theorem now tells us that $H$ contains a normal $q$-complement, which is impossible since $H$  simple. This contradiction shows that $R(G)=Z(G)$, as claimed. 

Taking into account Lemma \ref{crquotientofsemisimple} observe that $R(G)\neq1$. Let $M$ be a soluble minimal normal subgroup of $G$. The minimality of $|G|$ implies that $G'M=HM$. Since $H<G'$, it follows that $M\leq G'$ and $G'=H\times M$. Since $M$ is central, the order of $M$ is a prime, say $p$. The minimality of $|G|$ implies that $M$ is a unique minimal soluble normal subgroup.  We deduce that $Z(G)$ is a $p$-group. Moreover, as $M\leq Z(G)$ we have
$G'/H\leq Z(G/H)$, that is, $G/H$ is nilpotent of class 2. 
	
Set $\overline{G}=G/Z(G)$. Observe that $R(\overline{G})=1$ and so by Lemma \ref{crquotientofsemisimple} $\overline{G'}=\overline{H}$. Moreover, $\overline{G}/\overline{H}$ is either cyclic or the Klein four-group. 
	
If $\overline{G}/\overline{H}$ is cyclic, then there is an element $g\in G$ such that $G=\langle g\rangle Z(G)H$. In this case $G/H$ is abelian and so $G'=H$, a contradiction. 
	
Therefore we assume that $\overline{G}/\overline{H}$ is the four-Klein group. It follows from Lemma \ref{outautogroup} that either $H\cong A_6$ or $H\cong PSL(3,4)$. 	
Let $a,b\in G$ be elements for which $G=\langle a,b\rangle Z(G)H$. Here $[a,b]\notin H$ since $G/H$ is nonabelian. Thus, the equality $G'=H\times M$ implies that $[a,b]$ is a $p$-element. This happens for every choice of $a$ and $b$. On the other hand, $G/H$ is nilpotent of class 2. Since $\overline{G}/\overline{H}$ is the Klein four-group, it follows that $p=2$. We can replace the element $b$ by $bh$, where $h\in H$. Therefore $[a,bh]$ is a 2-element for any $h\in H$.
	
If $H\cong A_6$, by Lemma \ref{a6case} $a$ and $b$ can be chosen in such a way that there exists $h\in H$ with $[\overline{a},\overline{b}\overline{h}]$ of odd order. This rules out the case $H\cong A_6$.

Therefore $H$ is isomorphic to $PSL(3,4)$. According to \cite{luma} in this case $Aut\, H$ splits over $H$, that is, $\overline{G}=\overline{H}\rtimes \langle\overline{a},\overline{b}\rangle$ with $\langle\overline{a},\overline{b}\rangle$ isomorphic to the Klein group. The Baer-Suzuki theorem \cite[Theorem 3.8.2]{gorenstein2007finite} guarantees that there is an element $\overline g\in \overline G$ such that $\langle \overline{a},\overline{a}^{\overline{g}}\rangle$ is not nilpotent. Therefore there exists an  odd order element $\overline{y}\in \langle \overline{a},\overline{a}^{\overline{g}}\rangle$. Since $\overline{G}/\overline{H}$ is a 2-group, it follows that $y$ can be chosen in $H$. Observe that $\langle \overline{a},\overline{a}^{\overline{g}}\rangle$ is a dihedral group, whence $\overline{y}^{\overline{a}}=\overline{y}^{-1}$. Consequently, $1\neq[\overline{y},\overline{a}]=\overline{y}^{-2}$ has odd order. On the other hand, $[a,b]\in Z(G)$ is a $2$-element. The equality $[a,by]=[a,y][a,b]$ implies that $[a,y]$ is a 2-element. This is a contradiction. The proof is complete. 
\end{proof} 
\begin{corollary}\label{radglinhaasaderiv}
	Let $G$ be an insoluble CPPO-group. Then 
	$R(G')=[G',R(G)]$.
\end{corollary}
\begin{proof} By Proposition \ref{derivsubofnonsol} $G'$ is a perfect group. Thus, $G'/R(G')$ is a nontrivial perfect CPPO-group with trivial soluble radical. Corollary \ref{perfectplustrivialradical} shows that $G'/R(G')$ is a simple group. On the other hand, observe that $G'/[G',R(G)]$ is a quasisimple CPPO-group. Lemma \ref{quasisimpleissimple} shows that $G'/[G',R(G)]$ is simple and so $R(G')=[G',R(G)]$. 
\end{proof} 
The following lemma will be useful.
\begin{lemma}\label{solubleperfect} Let $G$ be a perfect group and $N\leq G$ a soluble normal subgroup such that $G/N$ is simple. Let $Q$ be a nontrivial subgroup of $G$ such that $Q\not\leq N$. Then $G=[G,Q]$.
\end{lemma}
\begin{proof} Note that $[G,Q]$ is a normal subgroup of $G$. Since $QN/N\neq1$ and $G/N$ is simple, it follows that $G=[G,Q]N$. We deduce that $G/[G,Q]$ is a perfect soluble group, that is $G/[G,Q]=1$. Hence, $G=[G,Q]$. 
\end{proof}
\begin{proposition}\label{thmnonsolubleCPPO} Let $G$ be an insoluble CPPO-group. Then $R(G')\leq O_2(G)$ and $G'/R(G')$ is a simple EPPO-group. 
\end{proposition}
\begin{proof}
We already know that $G'/R(G')$ is a simple EPPO-group. Therefore, it suffices to show that $R(G')\leq O_2(G)$. Assume that this is false and let $G$ be a counterexample of minimal order. By Proposition \ref{derivsubofnonsol} $G'$ is a perfect group so the minimality of $|G|$ shows that $G$ is perfect. It follows that $G/R(G)$ is a nonabelian simple group.
	
Let $M$ be a soluble minimal normal subgroup of $G$. Then $M$ is an elementary abelian $p$-subgroup for some prime number $p$. The minimality of $|G|$ yields that $R(G)/M$ is a 2-group and, in particular, $p\neq2$. Hence, $M=O_{p}(G)$ is a unique soluble minimal normal subgroup of $G$. In particular, $M=F(G)$. 
	
Suppose first that $M=R(G)$. Lemma \ref{existelemabelqsub} says that there exists an elementary abelian 2-subgroup $Q$ and an element $a\in G$ of odd prime power order such that $Q=[Q,a]$. By virtue of Lemma \ref{solubleperfect} we deduce that $G=[G,Q]$. Consequently, $Q$ does not centralize $R(G)$, otherwise we would have $R(G)=Z(G)$ and Lemma \ref{quasisimpleissimple} would imply that $R(G)=1$, a contradiction. Observe that $(\langle a\rangle,Q,R(G))$ is a tower with abelian factors. Moreover, $Q$ is an elementary abelian noncyclic group. Combining this with Lemma \ref{aaa} we deduce that $G$ has a commutator whose order is not a prime power, a contradiction. We conclude that $M$ is a proper subgroup of $R(G)$ and hence $R(G)/M$ is a nontrivial 2-group. 
	
Since $O_2(G)=1$, the subgroup $M$ is not central in $R(G)$. The minimality of $M$ implies that $M=[R(G),M]=[G,M]$. Again, using that $G$ is perfect and that $M\nleq Z(G)$ we conclude that $C_G(M)=M$.
	
Let $q\in\pi(G)\setminus\{2,p\}$. By Lemma \ref{existelemabelqsub}, there exists an elementary abelian $q$-subgroup $Q$ and a prime power order element $a\in G$, whose order is prime to $q$, such that $Q=[Q,a]$. Set $L=\langle a\rangle QR(G)$. Since $M$ is self-centralizing, it follows that $(\langle a\rangle,Q,M)$ is a tower of $L$. Now, Lemma \ref{heightequaltower} and Theorem \ref{main1} show that $h(L)=3$. 
	
Taking into account that $Q=[Q,a]$ we observe that $Q\leq \gamma_\infty(L)\leq F_{2}(L)$. Moreover, $M\leq F(L)$ and so $F(L)$ is a $p$-group. Since $R(G)/F(L)$ is a 2-group, $R(G)\leq F_2(L)$. Here $F_2(L)$ denotes the second term of the upper Fitting series of $L$. Let $S$ be a Sylow 2-subgroup of $R(G)$. The subgroups $Q$ and $S$ are both contained in $F_2(L)$ and so $[Q,S]\leq F(L)$. Hence, $[Q,S]\leq M$. This means that $Q$ centralizes $R(G)$ modulo $M$. Keeping in mind that in view of Lemma \ref{solubleperfect} $G=[G,Q]$ we now deduce that $G$ centralizes $R(G)$ modulo $M$. In other words, we proved that $R/M=Z(G/M)$. By Lemma \ref{quasisimpleissimple} we conclude that $R/M=1$, that is, $R=M$, a contradiction. The proof is complete. 
\end{proof}

Remark that the combination of Proposition \ref{derivsubofnonsol}, Corollary \ref{radglinhaasaderiv}, and Proposition \ref{thmnonsolubleCPPO} establishes Theorem \ref{main2}.


\begin{thebibliography}{b}


\bibitem{tiedt} W. Bannuscher, G. Tiedt, On a theorem do Deaconescu, Rostok. Math. Kolloq., \textbf{47}  (1994), 23--26.

\bibitem{brandl} Brandl, Rolf, {Finite Groups all of whose elements are of prime power order}. \textit{Boll. Un. Mat. Ital. A(5)}
\textbf{18} (1981) 491--493.


\bibitem{Casolo} C. Casolo, E. Jabara,  P. Spiga, On the Fitting height of factorised soluble groups, J. Group Theory, \textbf{17}  (2014), 911--924. 

\bibitem{delgado} A. Delgado, Yu-Fen Wu, On locally finite groups in which every elements has prime power order, Illinois J. Math., \textbf{46}  (2002), 885--891.

\bibitem{gorenstein2007finite} D. Gorenstein, \textit{Finite Groups}, Chelsea Publishing Company, New York, 1980.

\bibitem{higman} G. Higman, Finite Groups in Which Every Element Has Prime Power Order, J. Lond. Math. Soc., \textbf{s1-32}  (1957), 335--342.

\bibitem{Oreconject} M.W. Liebeck, E.A. O'Brien, A. Shalev, Tiep Pham Huu, The Ore Conjecture,  J. Eur. Math. Soc., \textbf{12}  (2010), 929--1008.

\bibitem{quasisimple} M.W. Liebeck, E.A. O'Brien, A. Shalev, Tiep Pham Huu, Commutators in finite quasisimple groups, Bull. Lond. Math. Soc.,
\textbf{43}  (2011), 1079--1092.

\bibitem{luma} A. Lucchini, F. Menegazzo, M. Morigi, On The existenceof a complement for a finite simple group in its automorphism group, Illinois J. Math., \textbf{47}  (2003), 395--418.

\bibitem{shu} P. Shumyatsky, Profinite groups in which many elements have prime power order, J. Algebra, \textbf{562} (2020), 188--199.
\bibitem{suzuki0}  M. Suzuki,  Finite groups with nilpotent centralizers. Trans. Amer. Math. Soc., {\bf 99}  (1961), 425--470.
\bibitem{Suzuki}  M. Suzuki, On a Class of Doubly Transitive Groups, Ann. Math.,  \textbf{75}  (1962), 105--145.

\bibitem{Turull} A. Turull,  Fitting height of groups and of fixed points, J. Algebra, \textbf{86}  (1984), 555--566.
\bibitem{wilson} R.A. Wilson, Robert A., \textit{The Finite Simple Groups}, Springer, London, 2009.
\end{thebibliography}
\end{document}